%% file: StochForcing.tex
\documentclass[11pt]{article}
\usepackage{fullpage}
\usepackage{amsfonts}
\usepackage[centertags]{amsmath}
\usepackage{amsthm}
\usepackage{graphicx}
\usepackage{subfigure}
\usepackage[usenames]{color}

\usepackage{ifthen}
\usepackage{color}
\newboolean{FinalVersionWithoutNotes}
\newcommand{\NNF}[1]{\ifthenelse{\boolean{FinalVersionWithoutNotes}}{}{\colorbox{yellow}{#1}}}
\setboolean{FinalVersionWithoutNotes}{true}
\newcommand{\mylabel}[1]{\label{#1} \NNF{$\;\;\;$ #1}  }

\theoremstyle{plain}
\newtheorem{thm}{Theorem}[section]

\newtheorem{lem}[thm]{Lemma} 

\theoremstyle{definition}
\newtheorem{defn}{Definition}[section]

\theoremstyle{remark}
\newtheorem{rem}{Remark}[section]
\numberwithin{equation}{section}

\newcommand{\ats}{\langle {\bf a} , {\bf s} (t)\rangle}

\newcommand{\ip}[2]{\langle #1 , #2 \rangle}
\newcommand{\rp}[1]{\mbox{Re}\left[ #1 \right]}

\newcommand{\leig}{\mbox{\boldmath{$\phi$}}  }
\newcommand{\reig}{\mbox{\boldmath{$\psi$}}  }

\newcommand{\adj}[1]{\mbox{adj}\left( #1 \right)}

\input{ss.tex}

\begin{document}
\title{On the Stability of Stochastic Parametrically  Forced  Equations with Rank One Forcing}
\author{Timothy\ Blass$^{1}$\and  L.A.\  Romero $^{2}$
\and J.R. Torczynski $^{3}$}
\footnotetext[1]{ Department of Mathematical Sciences, Carnegie Mellon University,
Pittsburgh, PA 15213, USA
{\tt tblass@andrew.cmu.edu}}
\footnotetext[2]
{Computational Mathematics and Algorithms Department,
Sandia National Laboratories, MS 1320, P.O. Box 5800, Albuquerque, NM 87123-1320,
{\tt lromero@sandia.gov}}
\footnotetext[3]
{Fluid Sciences and Engineering Department,
Sandia National Laboratories, MS 0840, P.O. Box 5800, Albuquerque, NM 87185-0840,
{\tt jrtorcz@sandia.gov}}
\footnotetext[4] {Sandia National Laboratories is a multi-program laboratory managed and operated by Sandia Corporation, a wholly owned subsidiary of
Lockheed Martin Corporation, for the U.S. Department of Energy's National
Nuclear Security Administration under contract DE-AC04-94AL85000.}

\maketitle
\abstract{We derive simplified formulas for analyzing the stability 
of stochastic parametrically forced linear systems.  
This extends the results in \cite{blass_rom} where,
assuming the stochastic excitation is small, 
the stability of such systems was computed 
using a weighted sum of the extended power spectral density 
over the eigenvalues of the unperturbed operator. 
In this paper, we show how to convert this to a sum 
over the residues of the extended power spectral density. 
For systems where the parametric forcing term is a rank one matrix, 
this leads to an enormous simplification.
}


\section{Introduction}
\mylabel{sec:intro}

In \cite{blass_rom}, 
we analyzed stochastically forced ODEs of the form 
\begin{equation}
{\bf B}_0  \odone{{\bf x} }{t} =  \left( {\bf A} _0   +
 \epsilon   f(t)  {\bf A} _1 \right) {\bf x} ,
\mylabel{eq:x} 
\end{equation}
where $\epsilon$ is a small parameter, 
$f(t)$ is a random function, 
and ${\bf A} _0$ and ${\bf A} _1$ are $N \times N$ matrices.
In \cite{blass_rom}, it was assumed that 
${\bf B}_0 $ was the identity matrix ${\bf I}$, 
but in this paper we relax this assumption.  

In that paper, we assumed that $f(t)$ could be written as
$f(t) = \ats$ where  ${\bf a}$ is an $n$-dimensional vector, 
${\bf s}$ is the output of an $n$-dimensional vector
Ornstein-Uhlenbeck process, and $\langle \cdot ,\cdot\rangle$ 
is the standard inner product on $\mathbb{C}^n$.  
We analyzed the moment stability of this equation.  
That is,  we determined under what conditions the various moments 
of ${\bf x}(t)$ remain bounded in time.  
Though our derivation assumed that $f(t)$ could be derived from
a vector Ornstein-Uhlenbeck process, we showed that, 
to second order in $\epsilon$, our stability criterion 
depended only on the extended power spectral density of $f(t)$.  
Here, the extended power spectral density $G(z)$ of $f(t)$ is defined as

\begin{equation}   
G(z) = \int_0^\infty R (\tau) e^{-z\tau}\, d\tau, 
\mylabel{eq:Gdef}
\end{equation} 
where $R(\tau)$ is the autocorrelation function of $f(t)$.  

We now briefly review the results of \cite{blass_rom}; 
a more thorough review is given in \S \ref{sec:rev}.
We showed that, to second order in $\epsilon$, 
the condition for the second moment to remain bounded is 
\begin{equation}
\rp{\lambda} = \rp{\lambda_0 + \epsilon^2 \lambda_2} < 0,
\mylabel{eqn:lams}
\end{equation}
where
\begin{equation}
\lambda_0 = \sigma_p + \sigma_q
\end{equation}
and $\sigma_k$ are the eigenvalues of ${\bf A} _0$
(if the unperturbed system is stable, then $\rp{\sigma_k}<0$).
The eigenvalues $\sigma_p $ and $\sigma_q$ are chosen 
so as to maximize the real part of the sum. 
More precisely, we want to choose $\sigma_p$ and $\sigma_q$
so as to minimize the value of $\epsilon$ 
that makes the sum in Eqn. (\ref{eqn:lams}) positive.  
For given values of $\sigma_p$ and $\sigma_q$, 
the expression for $\lambda_2$ involves  terms of the form

\begin{equation}
I_p =  \sum_{k=1}^N \chi_{pk} \chi_{kp} G(\sigma_p-\sigma_k) ,
\mylabel{eqn:Ip}
\end{equation}
where 
$\chi_{ij}$ are coefficients that depend on 
inner products of the $j$th adjoint eigenvector of
${\bf A} _0$ with  ${\bf A} _1$ multiplied by the
$i$th eigenvector of  ${\bf A} _0$.

In this paper, we show that, if ${\bf A} _1$ is a rank one matrix 
(${\bf A} _1 = {\bf u} {\bf v} ^T$), 
then we can convert the sum in Eqn. (\ref{eqn:Ip}) 
to a sum over the residues of $G$.  
We show that, if ${\bf A} _1$ is rank one, 
then the characteristic equation for the eigenvalues of 
${\bf A} _0 + \epsilon {\bf A} _1$ can always be written as

\begin{equation}
\det\left( \sigma {\bf B}_0 - {\bf A}_0 - \epsilon {\bf A} _1 \right) =0
\;\;\mbox{ iff  $\, g_A(\sigma,\epsilon)=0$ },
\end{equation}
where 
\begin{equation}
g_A(\sigma,\epsilon) = f_A(\sigma) + \epsilon.
\mylabel{eqn:gA}
\end{equation}
In \S \ref{sec:rank1}, we give an explicit formula for $f_A(\sigma)$.  
For now, we merely note that in this case 
the sum $I_p$ can be written as

\begin{equation}
I_p =  \frac{1}{f_A'(\sigma_p) } \sum_{m} \frac{ 1  }{f_A( \sigma_p - \mu_m ) }
r_m.
\mylabel{eqn:Ipfp}
\end{equation}
Here, $\mu_m$ are the poles of the extended power spectral density
$G(z)$, and $r_m$ are the residues of $G$ at these poles.

The results in \cite{blass_rom}
were derived for ordinary differential equations. 
However, the stability criterion (up to second order) 
can be expressed in terms of quantities such as
eigenvalues, eigenfunctions, and inner products 
that carry over to partial differential equations.  
Hence, it is not unreasonable to expect that 
the results carry over to partial differential equations.  
If we make this assumption, then, 
in order to  apply the results in \cite{blass_rom}, 
we need to evaluate infinite sums. 
However, if our function $G$ has a finite number of poles, 
the results in this paper allow us to convert the infinite sum 
into a finite sum over the residues at the poles of $G$.  

The formulation  in this paper was 
arrived at by considering the stability of 
stochastically forced Faraday waves 
(i.e., standing waves in the liquid-gas interface 
in a vertically oscillating container). 
For small values of the liquid viscosity, 
it is possible to approximate the equation 
for the height of the free surface using the Mathieu equation. 
However, for more viscous waves, 
it is necessary to solve a partial differential equation 
to determine the height. 
We show that the formulation given in this paper
applies to this problem and enormously simplifies the calculations 
(cf. Remark \ref{rem:use}).   
However, the purpose of this paper is to present 
the mathematical formalism for doing this, 
not to investigate the parameter space in the Faraday wave problem. 
In a later paper, we will more fully discuss 
the problem of stochastic Faraday waves.  

In \S \ref{sec:rev}, we  briefly review the results from \cite{blass_rom}.
We present these results for general  
symmetric positive definite mass matrices ${\bf B} _0$ 
and justify this in Appendix A. 
In \S \ref{sec:convert},  
we show how to convert the sum over the eigenvalues of ${\bf A} _0$ 
in Eqn. (\ref{eqn:Ip}) to a sum over the poles of $G$ 
for arbitrary (not only rank one) matrices ${\bf A} _1$. 
This conversion is particularly simple when ${\bf A} _1$ 
is a rank one matrix, and we show how to do this in \S \ref{sec:rank1}.  
In \S \ref{sec:mech}, we discuss a simple mechanical example where
${\bf A} _1$ is rank one. 
In \S \ref{sec:far}, we show how these results 
apply to stochastically induced Faraday waves.  

\section{Summary  of Previous Work}
\mylabel{sec:rev}

In \cite{blass_rom}, we analyzed the stability of Eqn. (\ref{eq:x}) 
assuming that ${\bf B}_0$ was the identity matrix.
In this section we summarize those results.  
However, we present the results for the more general case 
where ${\bf B}_0$ is assumed to be a symmetric positive definite matrix.
The justification for applying the results in \cite{blass_rom} to
more general mass matrices is straightforward 
and is given in Appendix A. 

In analyzing the stability of Eqn. (\ref{eq:x}), 
we suppose that the pair $({\bf A} _0,{\bf B} _0)$ has 
generalized eigenvalues $\sigma_k$, $k=1,\ldots,N$, 
generalized eigenvectors $\leig_k$, $k=1,\ldots,N$ 
and adjoint eigenvectors $\reig_k$, $k=1,\ldots,N$:
\begin{equation}
{\bf A} _0 \leig  _k = \sigma_k {\bf B}_0  \leig _k, \;\;\;
{\bf A} _0 ^{\star} \reig _k =\overline{\sigma}_k {\bf B}_0 \reig _k,\quad k=1,\ldots,N ,
\mylabel{eqn:geneig}
\end{equation}
where the eigenvectors are normalized
with respect to the inner product
$\langle {\bf v}, {\bf u} \rangle_B = \overline{{\bf v} }^T {\bf B}_0  {\bf u}$,
so that
\begin{equation}
\langle \reig _i , \leig  _j \rangle_B = \delta_{ij}.
\mylabel{eqn:norm}
\end{equation}
In \cite{blass_rom}, 
we showed that, for given values of $\sigma_p$ and $\sigma_q$,  
the parameter $\lambda_2$ in Eqn. (\ref{eqn:lams}) could be written as 
\begin{equation}
\lambda_2=8\sum_{j,k=1}^N
\frac{C_{jkpq}C_{pqjk} }{1+\delta_{pq}}G(\sigma_p + \sigma_q -\sigma_j-\sigma_k),
\mylabel{eq:lam2}
\end{equation}
where
\begin{equation}
C_{jk\ell m} = \frac{1}{4}\left(\delta_{jm} \chi_{k\ell} +
\delta_{km}\chi_{j\ell} 
+\delta_{j\ell}\chi_{km} +\delta_{k\ell}\chi_{jm} \right),
\mylabel{eqn:Cijkl}
\end{equation}
\begin{equation}
\chi_{ij} = 
\ip{\reig _i}{{\bf A} _1  \leig  _j},
\mylabel{eqn:defchi}
\end{equation}
and $\langle \cdot , \cdot\rangle$ is the standard inner product
$\langle {\bf v}, {\bf u} \rangle = \overline{{\bf v} }^T {\bf u}$.




Though Eqn. (\ref{eq:lam2}) was derived under the assumption 
that $f(t)$ was the output from a vector Ornstein-Uhlenbeck process, 
the fact that our answer depends only on 
the extended power spectral density of the process strongly
suggests that we should be able to apply this formula more generally. 
We make this assumption in this paper and do not limit our analysis
to such processes.  



Though Eqn. (\ref{eqn:Cijkl}) is 
a compact way of writing the equation for $\lambda_2$, 
we prefer to expand the Kronecker delta functions in this equation to
get the alternative form
\begin{equation} 
 \lambda_2 
 =  \frac{1 }{2(1+\delta_{pq})}\Bigg( 4 \chi_{pp} \chi_{qq} G(0)+2 \chi_{pq} \chi_{qp} 
\left( G(\sigma_p-\sigma_q) + G(\sigma_q - \sigma_p)\right) + 
2 \left(I_p + I_q + 2 \delta_{pq} I_p  \right)   \Bigg),
\mylabel{eqn:lam2}
\end{equation} 
where $I_p$ is given as in Eqn. (\ref{eqn:Ip}).  

The main simplification in this paper comes from finding
an alternative expression for $I_p$
and an expression for the products of the $\chi_{ij}$
that does not require computing eigenvectors, adjoint eigenvectors,
and inner products.

\section{Converting the Sum $I_p$ }
\mylabel{sec:convert}
To compute $\lambda_2$, we must compute $I_p$ in Eqn. (\ref{eqn:Ip}),
where the sum is taken over all of the 
generalized eigenvalues of $({\bf A} _0,{\bf B}_0)$.
For large systems, this sum can be cumbersome to compute
because one needs to compute all of the eigenvalues, eigenvectors,
adjoint eigenvectors, and coefficients $\chi_{pk}\chi_{kp}$.
In the case of PDEs, the computation of $I_p$ is further complicated by
the fact that there are infinitely many eigenvalues, etc., 
so one needs to decide how to truncate the series without losing accuracy.
We derive an alternative expression for $I_p$ in two stages. 
First, we show that the coefficients $\chi_{pk}\chi_{kp}$
can be computed in terms of an auxiliary function
that determines an \emph{equivalent characteristic equation}, 
which we now define.
\begin{defn}
Given a pair of matrices $({\bf K},{\bf M})$, 
we say that a function $g(\sigma)$ determines
an equivalent characteristic equation for $({\bf K},{\bf M})$ provided 
$g(\sigma)=0$ if and only if $\det \left( \sigma {\bf M} - {\bf K} \right) =0$.  
We  refer to such functions $g$ as equivalent characteristic functions.
\end{defn}
An equivalent characteristic equation allows us to replace 
$\chi_{pk}\chi_{kp}$ in Eqn. (\ref{eqn:Ip}) 
with a more convenient expression (cf. Eqn (\ref{eqn:chipk})). 
The second stage is to then convert the expression for $I_p$ 
from a sum over the generalized eigenvalues of $({\bf A} _0,{\bf B}_0)$ 
into a sum over the poles of $G$. 
This is done assuming $G$ is meromorphic 
and $f'_p(\sigma_k)\neq 0$ for each $k$, 
so that an argument involving contour integration can be applied. 

\subsection{Alternative Expression for $\chi_{pk}\chi_{kp}$}
Using the standard inner product, we can write 
\begin{equation}
\chi_{pk} \chi_{kp}=
\overline{\mbox{\boldmath{$\psi$}}} _k^T {\bf A} _1   \mbox{\boldmath{$\phi$}} _p
\overline{\mbox{\boldmath{$\psi$}}} _p^T {\bf A} _1 \mbox{\boldmath{$\phi$}} _k
= \langle \mbox{\boldmath{$\psi$}} _k, {\bf L} _p \mbox{\boldmath{$\phi$}} _k \rangle,
\mylabel{eqn:chiprod}
\end{equation} 
where
\begin{equation}
{\bf L} _p = {\bf a}_p {\bf b}_p ^T
\end{equation}
and
\begin{equation}
{\bf a}_p = {\bf A} _1 \mbox{\boldmath{$\phi$}} _p,\qquad
{\bf b}_p^T = \overline{ \mbox{\boldmath{$\psi$}} } _p^T {\bf A} _1.
\mylabel{eqn:apbp}
\end{equation}
We define
\begin{equation}
{\bf A} _p (\epsilon) = {\bf A} _0 + \epsilon {\bf L} _p 
= {\bf A} _0 + \epsilon {\bf a} _p {\bf b} _p^T
\mylabel{eqn:Ap}
\end{equation}
and denote the generalized eigenvalues
and eigenvectors of $({\bf A}_p,{\bf B}_0)$ by $\hatsig_k$ and $\hat\leig_k$.
That is, $\hatsig_k(\epsilon)$ is the eigenvalue satisfying
\begin{equation}
{\bf A} _p(\epsilon)  \mbox{\boldmath{$\hatphi$}} _k = 
\hatsig_k  {\bf B} _0 \mbox{\boldmath{$\hatphi$}} _k .
\mylabel{eqn:eigeps}
\end{equation}
Noting that, at $\epsilon=0$, 
$\hat \reig_p(0) = \reig_p$, $\hat \leig_p(0) = \leig_p$,  we can
differentiate Eqn. (\ref{eqn:eigeps}) with respect to $\epsilon$, 
evaluate at $\epsilon = 0$, left-multiply by
$\overline{\reig}^T_p$ and use
$\overline{\reig}^T_p(\sigma_p {\bf B}_0-{\bf A}_0)=0$
to obtain the well known result from the perturbation theory of eigenvalues
\begin{equation}
  \odone{\hatsig_k}{\epsilon} \Big|_{\epsilon=0}
= \frac{\overline{\reig}^T_p{\bf L}_p\leig_p}{\overline{\reig}^T_p{\bf B}_0\leig_p}.
\mylabel{eqn:dsighat}
\end{equation}
Combining Eqn. (\ref{eqn:dsighat})
with Eqn. (\ref{eqn:chiprod}) 
and the normalization in Eqn. (\ref{eqn:norm}), we have
\begin{equation}
\chi_{pk} \chi_{kp} = \odone{\hatsig_k}{\epsilon} \Big|_{\epsilon=0},
\mylabel{eqn:defgam}
\end{equation}
so we can write 
the  sum in Eqn. (\ref{eqn:Ip})   as
\begin{equation}
I_p = \sum_{k=1}^N \left(\odone{\hatsig_k}{\epsilon} \Big|_{\epsilon=0}\right)
G( \sigma_p - \sigma_k).
\end{equation}
If the eigenvalues  $\hatsig(\epsilon)$ satisfy  the  equivalent characteristic equation 
\begin{equation}
g_p(\hatsig,\epsilon)=0
\mylabel{eqn:gp0}
\end{equation}
for some  equivalent characteristic function $g_p$,
then implicitly differentiating Eqn. (\ref{eqn:gp0}), evaluating the
result at $\epsilon=0$, and solving for $d \hatsig/d \epsilon$ gives 
\begin{equation}
\odone{\hatsig_k}{\epsilon} \Big|_{\epsilon =0}= - \frac{ \pdone{g_p}{\epsilon}(\sigma _k,0)}
{\pdone{g_p}{\sigma}(\sigma_k,0 ) }.
\mylabel{eqn:dsigh}
\end{equation}
On the right  side of Eqn. (\ref{eqn:dsigh}), we are using
$\hatsig_k(0)= \sigma_k$, where $\sigma_k$ are the generalized 
eigenvalues of $({\bf A} _0,{\bf B} _0)$.

The following lemma shows that, for matrices of the
form  ${\bf A} _p(\epsilon) $ as in Eqn. (\ref{eqn:Ap}), 
there is always an equivalent characteristic equation
with a particularly simple form.  
\begin{lem}
The eigenvalues
$\hatsig(\epsilon)$ of ${\bf A} _p(\epsilon)$ 
in Eqn. (\ref{eqn:Ap}) satisfy the equivalent characteristic equation
\begin{equation}
g_p(\hatsig,\epsilon): = f_p(\hatsig) + \epsilon =0,
\mylabel{eqn:g}
\end{equation}
where
\begin{equation}
\frac{1}{f_p(\hatsig) } =- {\bf b}_p ^T \left( \hatsig {\bf B}_0 - {\bf A} _0
\right)^{-1} {\bf a}_p .
\mylabel{eqn:fp}
\end{equation}
\mylabel{lem:disp}
\end{lem}
\begin{rem}
For any $\hat\sigma$ that is not a generalized eigenvalue
of $({\bf A} _0,{\bf B} _0)$, 
the formula in Eqn. (\ref{eqn:fp}) is well-defined.
In fact, $f_p$ is well-defined even at each generalized eigenvalue $\sigma_k$
where $f_p(\sigma_k)=0$. Denoting the adjugate
matrix (the transpose of the cofactor matrix)  of $\hatsig {\bf B}_0 - {\bf A} _0$ by $\adj{\hatsig {\bf B}_0 - {\bf A} _0}$,
Cramer's rule gives $\adj{\hatsig {\bf B}_0 - {\bf A} _0} = \det(\hatsig {\bf B}_0 - {\bf A} _0)
(\hatsig {\bf B}_0 - {\bf A} _0)^{-1}$ whenever $\hatsig {\bf B}_0 - {\bf A} _0$ is non-singular.
If we take $f_p(\hatsig) = -\det(\hatsig {\bf B}_0 - {\bf A} _0)/({\bf b}_p^T
\adj{\hatsig {\bf B}_0 - {\bf A} _0}{\bf a}_p)$ as the definition of $f_p$, then it agrees
with Eqn. (\ref{eqn:fp}) for any $\hatsig \neq \sigma_k$, and shows that $f_p(\sigma_k)=0$.
\end{rem}
\begin{proof}
The eigenvalues $\hatsig$ satisfy
\begin{equation}
h_p(\hatsig,\epsilon) = 
\det \left( \hatsig {\bf B}_0 - {\bf A} _0 - \epsilon {\bf a}_p {\bf b}_p ^T \right)
=0.
\end{equation}
If we define
\begin{equation}
h_0(\hatsig) := \det\left( \hatsig {\bf B}_0 - {\bf A} _0 \right) ,
\end{equation}
then using the  formula for the determinant of the rank one update of the identity matrix, 
\begin{equation}
\det\left( {\bf I} +  {\bf u} {\bf v} ^T \right) = 1 + {\bf v} ^T {\bf u}, 
\end{equation}
we see that the eigenvalues satisfy
\begin{equation}
0=h_p(\hatsig,\epsilon) = h_0(\hatsig) \left( 1 - \epsilon
{\bf b}_p ^T \left( \hatsig {\bf B}_0 - {\bf A} _0 \right)^{-1} {\bf a}_p \right)
=  h_0(\hatsig) + \epsilon  \frac{h_0(\hatsig)}{f_p(\hatsig)} ,
\mylabel{eqn:gp}
\end{equation}
where
$f_p(\hatsig)$ is defined as in Eqn. (\ref{eqn:fp}).  Note that 
using Cramer's rule  makes it  simple to see that 
$h_0(\hatsig)/f_p(\hatsig)$ is never singular, so, if we divide $h_p(\hatsig)$ by
this quantity,  we do not introduce any new zeros.  Dividing Eqn. (\ref{eqn:gp}) by
$h_0(\hatsig)/f_p(\hatsig)$ yields the equation for $\hatsig$  
as in Eqn. (\ref{eqn:g}).  
\end{proof}

Implicitly differentiating  equation (\ref{eqn:g})
with 
respect to $\epsilon$ and evaluating the result at 
$\epsilon=0$ and $\sigma = \sigma_k$ we find that
\begin{equation}
\odone{\hatsig_k}{\epsilon} \Big|_{\epsilon=0}=
-\frac{ 1 }{ f_p'(\sigma_k) }.
\end{equation}
Combining this with 
Eqn.  (\ref{eqn:defgam}) gives the  following lemma.
\begin{lem}
With $\chi_{ij}$ defined as in Eqn. (\ref{eqn:defchi}), and $f_p$ as in
Eqn. (\ref{eqn:fp}), we have
\begin{equation}
\chi_{pk} \chi_{kp} = - \frac{ 1}{ f'_p(\sigma_k)}.
\mylabel{eqn:chipk}
\end{equation}
\mylabel{lem:chipk}
\end{lem}
The fact that $g_p$ depends linearly on $\epsilon$ is due to the
fact that ${\bf A}_p$ is a rank one update of ${\bf A}_0$. 
The choice of ${\bf A}_p$ is motivated by the form of $\chi_{pk}\chi_{kp}$
in Eqn. (\ref{eqn:chiprod}) and is  related to the original problem
Eqn. (\ref{eq:x}) only through the stability analysis (and specifically, the form of $I_p$).
No assumptions have been made on ${\bf A}_1$ at this point, 
so the formulas for $g_p$ and $f_p$ hold in general. 
The only obstacle to using Eqn. (\ref{eqn:chipk}) 
to compute $\chi_{pk}\chi_{kp}$
is that $f_p$ (and $f'_p$) may be difficult to compute.
In particular, we need to compute the 
eigenvector and the adjoint eigenvector of $\sigma_p$ and take some inner products.  
In \S \ref{sec:rank1},
we show that if ${\bf A}_1$ is rank one, then the same algebraic manipulations
can be performed to find an equivalent characteristic function $g_A$
that is simple to compute
(i.e., it does not involve the generalized eigenfunctions and adjoint eigenfunctions)
and leads to a simple  formula for $f_p$.

\subsection{Rewriting $I_p$ as a Sum over the Poles of $G$}
Using Lem. \ref{lem:chipk} in Eqn. (\ref{eqn:Ip}), we can write
\begin{equation}
I_p = - \sum_{k=1}^N \frac{ 1 } {f_p'(\sigma_k)} 
G(\sigma_p - \sigma_k).
\end{equation}
Cramer's rule can be used to show that 
$1/f_p(\sigma)$ decays at least as fast as $1/\sigma$ as 
$|\sigma |$ approaches $\infty$.  
The form of the function
 $G(\sigma)$ in Eqn. \eqref{eqn:GS} shows that it also decays at
 least as fast as $1/\sigma$ for large values of $|\sigma |$.  
With this in mind,
\begin{equation}
  \label{lim_int}
  0 = \lim_{R\to \infty} \frac{1}{2\pi i}\int_{C_R}\frac{G(\sigma_p-\sigma)}
{f_p(\sigma)} \, d\sigma,
\end{equation}
where  $C_R:=\{\sigma \in \mathbb C \, : \, |\sigma|=R\}$. 
We denote the poles of $G(z)$ by $\mu_m$, $m=1,\ldots,M$ 
and the residue of $G$ at $\mu_m$ by $r_m$, for each $m$.
Thus, $-r_m$ is the residue of $G(\sigma_p-\sigma)$ at $\sigma_p-\mu_m$. 
Now, applying the Residue theorem to \eqref{lim_int}, and recalling that
the zeros of $f_p$ are $\sigma_k$, we see that
\begin{equation}
  \label{residues}
  0 =  \sum_{k=1}^N \frac{ G(\sigma_p - \sigma_k) } {f_p'(\sigma_k)} 
 - \sum_{m=1}^M\frac{r_m}{f_p(\sigma_p-\mu_m)}
\end{equation}
assuming $f'_p(\sigma_k)\neq 0$ for each $k$. 
This leads to the following theorem.



\begin{thm}
Let $f_p(\sigma)$ be defined as in  Eqn. (\ref{eqn:fp}). 
If $G$ is meromorphic and $f'_p(\sigma_k)\neq 0$ for each $k$,
then the sum
$I_p$ in Eqn. (\ref{eqn:Ip}) can be written as
\begin{equation}
I_p = - \sum_{m} \frac{ r_m  }{f_p( \sigma_p - \mu_m ) }.
\mylabel{eqn:Ipfp0}
\end{equation}
Here, the sum is taken over all of the poles $\mu_m$ of $G(z)$, and $r_m$ is the residue of
$G(z)$ at the pole $\mu_m$.
\mylabel{thm:one}
\end{thm}
Thm. \ref{thm:one} provides a simple expression for $I_p$ provided
there is a practical way to compute the function $f_p$. 
We show in the next section that when ${\bf A}_1$ is rank one,
then it is simple to compute $f_p$, and hence it is simple
to compute $\lambda_2$ and determine the moment stability of the system. 
\begin{rem}
We note that Thm. \ref{thm:one} applies to any system
of the form \eqref{eq:x}, including systems that
are a priori unstable (i.e., if $f_p(\sigma)=0$ has
solutions in the right half-plane). 
However, this scenario is perhaps less interesting because in order
to stabilize a system with stochastic forcing, one would expect to 
use a large value of $\epsilon$,  requiring a different approach 
than presented here.
We are interested in applying Thm. \ref{thm:one} to systems that
are stable for $\epsilon =0$ (i.e., $\rp{\sigma_k}<0$ for all $k$),
specifically, the Faraday wave problem discussed in \S  \ref{sec:far}.
\end{rem}

\section{Case where ${\bf A} _1$ is Rank One}
\mylabel{sec:rank1}
We now 
suppose that ${\bf A} _1$ is rank one and hence has   the form
\begin{equation}
{\bf A} _1 = {\bf u} {\bf v} ^T.
\mylabel{eqn:rank1}
\end{equation}
Thus, ${\bf A} _0 + \epsilon {\bf A} _1 $ is a rank one update of
${\bf A} _0$, just like ${\bf A} _p$ was in Eqn. (\ref{eqn:Ap}). The same
algebraic manipulations that were applied to $\det(\hatsig {\bf B}_0-{\bf A}_p)$
in the proof of Lem. \ref{lem:disp} yields 
 an equivalent characteristic equation for
  $\det \left(  \sigma {\bf B}_0 - {\bf A} _0 - \epsilon {\bf A} _1  \right)$ 
given by 
 $g_A(\sigma,\epsilon)$,
where  $g_A(\sigma)$ is defined as in Eqn. (\ref{eqn:gA}) and 
\begin{equation}
\frac{1}{f_A(\sigma) } = - {\bf v} ^T \left( \sigma {\bf B} _0 - {\bf A} _0  \right)^{-1}
{\bf u} .
\mylabel{eqn:fA}
\end{equation}
The following lemma shows that we can express the function $f_p(\sigma)$ in
Eqn. (\ref{eqn:fp}) in 
terms of $f_A(\sigma)$.  
\begin{lem}
Assuming ${\bf A} _1 = {\bf u} {\bf v} ^T$, then the function
$f_p(\sigma)$
in Eqn. (\ref{eqn:fp}) can be written as 
\begin{equation}
f_p(\sigma) = -f_A(\sigma) f'_A(\sigma_p).
\mylabel{eqn:fpfA}
\end{equation}
\mylabel{lem:rank1}
\end{lem}
\begin{proof}
Using Eqns. (\ref{eqn:fp}) and (\ref{eqn:apbp})
 and our expression in Eqn. (\ref{eqn:rank1})  for ${\bf A} _1$, 
a simple calculation shows that 
\begin{equation}
f_p(\sigma) = \frac{1}{ 
\left( \overline{ \psi}_p^T {\bf u} {\bf v} ^T \mbox{\boldmath{$\phi$}}_p \right) }
f_A(\sigma).
\end{equation}
To prove the theorem, it is only necessary to show that 
\begin{equation}
\overline{\reig}_p^T {\bf u} {\bf v} ^T \leig_p = 
- \frac{1}{f'_A(\sigma_p)}.
\mylabel{eqn:chipp}
\end{equation}
Since 
$({\bf A}_0+\epsilon{\bf A}_1)\leig_p (\epsilon)= \sigma_p(\epsilon) {\bf B}_0\leig_p(\epsilon)$,
if we differentiate and left-multiply by $\overline{\reig}_p^T $,
we see that 
\begin{equation} 
\frac{d}{d\epsilon}\sigma_p\big|_{\epsilon=0} = 
\overline{\reig}_p^T{\bf A}_1  \leig_p =\overline{\reig}_p^T {\bf u} {\bf v} ^T \leig_p
\end{equation} 
(analogous to Eqn. (\ref{eqn:dsighat})).
Since $\sigma_p(\epsilon)$ 
satisfies Eqn. (\ref{eqn:gA}),
the formula in Eqn. (\ref{eqn:chipp}) now follows by 
implicitly differentiating Eqn. (\ref{eqn:gA}).  
\end{proof}

Combining 
Thm. \ref{thm:one} with Lem. \ref{lem:rank1} yields  the following
theorem.  

\begin{thm}
If ${\bf A} _1$ is a rank one matrix
(and hence ${\bf A} _1 = {\bf u} {\bf v} ^T$), then
the roots of $\det\left(\sigma {\bf B}_0 - {\bf A} _0 - \epsilon {\bf A} _1 \right)$
are the same as the roots of $g_A(\sigma) = f_A(\sigma) + \epsilon$, where
 $f_A(\sigma)$ is defined as in  Eqn. (\ref{eqn:fA}).  
Furthermore, under the assumptions of Thm. \ref{thm:one},
the sum $I_p$ in Eqn. (\ref{eqn:Ip}) can be written as
in Eqn. (\ref{eqn:Ipfp}).
Here, the sum is taken over all of the poles $\mu_m$ of $G(z)$ and $r_m$ is the residue of
$G(z)$ at the pole $\mu_m$.
\mylabel{thm:two}
\end{thm}

In order to evaluate 
$\lambda_2$ using Eqn. (\ref{eqn:lam2}), it is necessary to  evaluate
$\chi_{pq} \chi_{qp}$, and $\chi_{pp} \chi_{qq}$.  
Using Lems. \ref{lem:chipk} and \ref{lem:rank1}, we see that
\begin{equation}
\chi_{pq}\chi_{qp} = \frac{ 1}{ f'_A(\sigma _q) f'_A(\sigma_p) }.
\mylabel{eqn:chiij}
\end{equation}
Using Eqn. (\ref{eqn:chipp}) and Eqn. (\ref{eqn:defchi}), we see that
\begin{equation}
\chi_{pp} = - \frac{1}{ f'_A(\sigma_p) }.
\mylabel{eqn:chipp2}
\end{equation}
Eqns. (\ref{eqn:chiij}) and (\ref{eqn:chipp2}) yield  the following
theorem.

\begin{thm}
Assuming ${\bf A} _1$ is a rank one matrix and that
$f_A(\sigma)$ is defined as in Eqn. (\ref{eqn:fA}) and
satisfies $f_A'(\sigma_k) \neq 0$ for $k=p,q$, we have
\begin{equation}
\chi_{pp} \chi_{qq} = \chi_{pq} \chi_{qp} = \frac{1}{f_A'(\sigma_q) f_A'(\sigma_p) }.
\end{equation}
\mylabel{thm:three}
\end{thm}
Thus, if  ${\bf A} _1$ is rank one, the terms in the
expression for $\lambda_2$ given in Eqn. (\ref{eqn:lam2})
can all be computed in terms of $f_A$, $\sigma_p$, $\sigma_q$,
and the poles and residues of $G$. Moreover, $f_A(\sigma)$ 
is straightforward to compute, as it
only requires knowing ${\bf B}_0$, ${\bf A}_0$, and ${\bf A}_1$.

\section{A Simple Mechanical System}
\mylabel{sec:mech}

In this section, we consider a simple mechanical example 
where ${\bf A} _1$ is rank one and hence allows us to use 
Thms. \ref{thm:two} and \ref{thm:three}.

We consider a pendulum attached to a support (the pivot) 
constrained to move along a horizontal line ($x$ direction) 
that is vibrated up and down vertically ($z$ direction). 
The motion of this system is described 
in the accelerated reference frame in which this line is fixed. 
In this frame, the system experiences an effective time-varying 
gravitational field given by the line's vertical acceleration. 
We denote the $x$ coordinate of the support by $\xi$ 
and the angle of the pendulum relative to 
the downward (negative) $z$ direction by $\theta$.  
The support has a mass of $m_S$, 
the bob at the end of the pendulum has a mass of $m_P$, 
and the pendulum shaft is massless and has a length of $\ell$. 
At any instant, the position of the support is denoted by ${\bf r}_S$, 
and the position of the pendulum bob is denoted by ${\bf r}_P$. 
These vectors are thus expressed as 
\begin{equation}
{\bf r}_S = \twovec{ \xi }{0},\qquad
{\bf r}_P = \twovec{ \xi}{0} - \ell {\bf e} _r(\theta),\qquad
{\bf e}_r = \twovec{ -\sin\theta }{ \cos\theta }.
\end{equation}
We  suppose that the support is attached to a linear spring
with spring constant $K_S$ and that the mass on the pendulum is
acted on by a  spatially uniform gravitational field with acceleration constant $g_0$.
In the reference frame of the pendulum, when a vertical forcing is applied,
the gravitational acceleration varies in time, which we write as
$g(t) = g_0 + \epsilon f(t)$.   
The Lagrangian of this system can be written as
\begin{equation}
\mathcal{L} = \frac{1}{2} m_S \dotxi^2 + \frac{1}{2} m_P
\left|  \dotxi {\bf e} _x - \ell  \dotthe \pdone{{\bf e} _r}{\theta} \right|^2
+ m_P \ell g(t) {\bf e} _z \cdot  {\bf e} _r(\theta) 
- \frac{1}{2} K_S \xi^2.
\end{equation}
Here, ${\bf e} _x$ and ${\bf e} _z$ are the unit vectors in
the $x$ and $z$ directions.
We  also assume that the spring attached to the support has damping 
$\gamma_S$ and that the pendulum has damping $\gamma_P$.  
The linearized equations of motion about the
equilibrium $\xi=0$, $\theta =0$ are given by 
\begin{equation}
{\bf M} {\bf \ddot z }  + {\bf C} {\bf \dot z }  + {\bf K} {\bf z} =0,\qquad
{\bf z} = \twovec{\xi}{\theta} ,
\mylabel{eqn:MCK}
\end{equation}
where
\begin{equation}
{\bf M} = \twomat{ m_S+ m_P}{\ell  m_P  }
                 {\ell m_P }{ \ell^2 m_P},\quad
{\bf C} = \twomat{ \gamma_S }{0}
                 { 0}{\gamma_P},\quad
{\bf K} = \twomat{ K_S}{0}{0}{g(t) \ell}.
\mylabel{eqn:K}
\end{equation}
Using $g(t) = g_0 + \epsilon f(t)$, we write ${\bf K} = {\bf K}_0 + \epsilon f(t) {\bf K}_1$,
where ${\bf K}_0 = {\rm diag}(K_S,g_0 \ell)$ and ${\bf K}_1 = {\rm diag}(0,\ell)$
are the unperturbed and perturbed parts of ${\bf K}$. 
Writing  Eqn. (\ref{eqn:MCK}) as a first-order system of equations, we obtain
a system as in Eqn. (\ref{eq:x}) with
\begin{equation}
{\bf B}_0 = \twomat{{\bf I}}{ {\bf 0}}{ {\bf 0} }{ {\bf M} },\quad
{\bf A}_0 = \twomat{ {\bf 0} }{ {\bf I} }{ -{\bf K}_0 }{ -{\bf C} },\quad
{\bf A}_1 = {\bf u}{\bf v}^T,
  \mylabel{eqn:first}
\end{equation}
where ${\bf u}=(0,0,0,-\ell)^T$ and ${\bf v} = (0,1,0,0)^T$.

Since ${\bf A}_1$ is rank one, we can apply the results of \S \ref{sec:rank1}.
Thus, we can compute $f_p$ by computing $f_A$ as in Eqn. (\ref{eqn:fA}).
Indeed, the  equivalent characteristic equation
\begin{equation}
h(\sigma) =\det(\sigma {\bf B}_0 - {\bf A}_0 - \epsilon {\bf A_1})=
 \det \left( \sigma^2 {\bf M} + \sigma {\bf C} + 
{\bf K}_0 + \epsilon {\bf K}_1 \right)=0
\end{equation}
 depends linearly on $\epsilon$ and
leads  to an explicit formula for $f_A$.  
In particular, we have
\begin{equation}
h(\sigma) = h_0(\sigma) + \epsilon h_1(\sigma)
\end{equation}
where
\begin{equation}
h_0(\sigma) := \sigma^4 \ell^2 m_P m_S + \sigma^3\left(  \gamma_P 
\left( m_P  + m_S \right) + \ell^2 \gamma_S m_P \right)
+ \sigma^2 \left( \gamma_S \gamma_P + K_s \ell^2 m_p \right)
+ \sigma \gamma_P K_S +g_0 h_1(\sigma),
\end{equation}
and
\begin{equation}
h_1(\sigma) := \sigma^2 \ell \left( m_P + m_S \right) + \sigma 
\gamma_S \ell + K_S \ell.
\end{equation}
Finally, we arrive at an equivalent characteristic equation
\begin{equation}
  h_A(\sigma) = f_A(\sigma) +  \epsilon=0,
\qquad \mbox{where}\quad
f_A(\sigma) = \frac{ h_0(\sigma) }{h_1(\sigma) }.
\end{equation}
We have computed $I_p$ for this system using 
Eqn. (\ref{eqn:Ip}), which involves computing  all of the
eigenvalues, eigenvectors, and coefficients $\chi_{pk}$
using inner products. We have also computed $I_p$ 
using Eqn. (\ref{eqn:Ipfp})
which involves computing the poles and residues of $G$,
$f_A(\sigma_p-\mu_m)$, and $f'_A(\sigma_p)$. 
These calculations were carried out using 
the extended power spectral density described in Appendix B.  
The two ways of doing the calculations  agree to within machine precision. 
Having an explicit formula for $f_A$, as above, allows one to 
avoid any computation involving the eigenfunctions and adjoint
eigenfunctions. The expression as a sum over the poles of $G$
allows one to avoid computing sums with many terms, which is the case when
Eqn. (\ref{eq:x}) is a large system. 
In this low-dimensional example, 
little effort is saved by using Eqn. (\ref{eqn:Ipfp}), 
but, for larger systems and for PDEs,
the formula for $I_p$ in Eqn. (\ref{eqn:Ipfp})
is much simpler to use than the formula in Eqn. (\ref{eqn:Ip}).

\section{Application to Stochastic Faraday Waves}
\mylabel{sec:far}

Here, we  briefly outline how Thms. \ref{thm:two} and \ref{thm:three} 
apply to the case of viscous capillary gravity waves 
and hence can be used to analyze stochastically induced Faraday waves.  

For small amounts of damping, 
it is possible to model Faraday waves 
using the damped Mathieu equation.
In \cite{zhang,repetto,berthet}, 
the stability of stochastically forced Faraday waves
was analyzed using the Mathieu equation and the results in \cite{kampen} 
for the stochastically forced Mathieu equation. 
In \cite{kumar,cerda2,cerda1,kumar2}, 
 deterministic viscous Faraday waves were analyzed
without making the small-damping approximation. 
We analyze the stability of 
stochastically forced Faraday waves 
without making the small-damping approximation 
but instead using the full linearized Navier-Stokes equations.

Thus, we consider the linearized equations 
for viscous capillary gravity waves.  
We assume that our system is in a time-varying gravitational field. 
The time variation arises from the fact that the container
holding the liquid is being moved up and down vertically 
and that the liquid motion is considered 
in a non-inertial frame of reference moving with the container.  

We Fourier-transform these equations in the horizontal direction, 
and, without loss of generality, 
we assume that our disturbance is two-dimensional, 
depending only on the horizontal coordinate $x$ 
and the vertical coordinate $z$.  
We assume that all quantities vary like
\begin{equation}
q(x,z,t) = q(z,t) e^{ i \alpha x}.
\end{equation}
With all of this in mind, 
the $x$ and $z$ components of the horizontally Fourier-transformed, 
linearized Navier-Stokes equations
with the hydrostatic pressure subtracted are given by 
\begin{equation}
\rho \pdone{u}{t} + i \alpha p  = \mu \left( \odtwo{u}{z} - \alpha^2 u \right),
\mylabel{eqn:momx0}
\end{equation}
\begin{equation}
\rho \pdone{w}{t} + \odone{p}{z}  = \mu \left( \odtwo{w}{z} - \alpha^2 w \right),
\mylabel{eqn:momz0}
\end{equation}
and the continuity equation is given by 
\begin{equation}
i \alpha u + \odone{w}{z}=0.
\mylabel{eqn:cont}
\end{equation}
Here  $(u,v,p)$ are the horizontal velocity, the vertical velocity, and
the pressure with the hydrostatic component removed.  

We apply a no-slip boundary condition at $z=-L$:
\begin{equation}
u(-L,t) =w(-L,t) =0.
\mylabel{eqn:bcbot}
\end{equation}
If $h(t)$ is the spatially Fourier-transformed height of the free surface,
the linearized kinematic boundary condition is given by 
\begin{equation}
\odone{h}{t} = w \;\;\;\;\mbox{ at $z=0$}.
\mylabel{eqn:kinem}
\end{equation}
The dynamic boundary conditions are given by 
\begin{equation}
2 \mu \odone{w}{z}- p = -\left( T  \alpha^2 +   
\rho \left( g_0 + \epsilon f(t) \right)  \right) h  
\;\;\;\;\mbox{ at $z=0 $} ,
\mylabel{eqn:bc1}
\end{equation}
\begin{equation}
\odone{u}{z} + i \alpha w =0  \;\;\;\;\mbox{ at $z =0 $}.
\mylabel{eqn:bc2}
\end{equation}
The function $f(t)$ in Eqn. (\ref{eqn:bc1}) corresponds to the
time-varying effective gravitational field 
caused by the vertical vibration of the container.  

If we were to spatially discretize this system in the $z$ direction, 
we would end up with a system of equations of the form
\begin{equation}
{\bf M} \odone{{\bf q}}{t} = {\bf K}_0  {\bf q}  +  \epsilon f(t) 
{\bf K} _1 {\bf q}, 
\mylabel{eqn:difalg}
\end{equation}
where ${\bf q}$ is a vector containing the discretization of 
$(u,w,p)$, along with $h$, and ${\bf M}$, ${\bf K} _0$, and ${\bf K} _1$ 
are matrices that depend on the wavenumber $\alpha$.  
Due to the continuity equation, the matrix ${\bf M}$ would be singular, 
thus yielding a system of differential algebraic equations
with KKT-like structure (as discussed in Appendix C.  
In Appendix C, we show that the results in \S \ref{sec:rev}
apply to such systems.

It is crucial to note that the dependence 
on the parameter $\epsilon$ is rank one. 
That is, all occurrences of the parameter $\epsilon $ 
in our governing equations multiply the unknown $h$. 
In our discretized equations of motion, 
this situation would yield the term
${\bf K} _1 = {\bf a} {\bf b} ^T $, 
where ${\bf b}$ is a vector that is
all zeros except for the component involving $h$. 
Thus, ${\bf K} _1$ is rank one.  

When the depth $L$ of the container is infinite, 
the dispersion relation for these equations is given by 
\begin{equation}
g_D(\sigma,\epsilon)= 
f_D(\sigma)  +  \epsilon,
\mylabel{eqn:fulldisp0}
\end{equation}
where

\begin{equation}
f_D(\sigma)= \frac{1}{\alpha}\left( \sigma + 2 \nu \alpha^2 \right)^2
- 4 \alpha^2  \left( \nu \right)^{3/2} \sqrt{ \sigma + \nu \alpha^2 }
+ \frac{ T}{\rho} \alpha^2 + g_0 . 
\mylabel{eqn:fulldisp}
\end{equation}
Since the dependence on $g$ appears as a rank one term 
in the governing equations,
this dispersion relation has the simple form 
that we expect to observe for a rank one system.  

For finite values of $L$, the dispersion relation is more complicated, 
but it still has the simple form required of a rank one system.  
We give this dispersion relation in Eqn. (\ref{eqn:finite}) in Appendix $D$.  

We have carried out calculations for stochastic Faraday waves
using both Thm. \ref{thm:two} and the formulation in \cite{blass_rom}.
We use the following parameter values: 
a liquid density of $\rho = 0.95\; {\rm g/cm^3}$, 
a liquid kinematic viscosity of $\nu = 0.1\; {\rm cm^2/s}$, 
a liquid-gas surface tension of $T = 70\; {\rm g/s^2}$, 
a steady gravitational acceleration of $g_0 = 1000\; {\rm cm/s^2}$, 
and a wavenumber of $\alpha = 5\; {\rm cm}^{-1}$. 
We consider the power spectral density described in Appendix B. 

Both of these formulations give identical results.
To illustrate this, Table \ref{table:one} 
shows how values from using Eqn. (\ref{eqn:Ip}) to evaluate $I_p$ 
compare to values from using Thm. \ref{thm:two} to evaluate this sum. 
In the table, we calculate the sum in Eqn. \eqref{eqn:Ip} for different values of $N$.
Clearly as $N$ approaches infinity, 
the values from Eqn. (\ref{eqn:Ip}) converge 
to the value from Thm. \ref{thm:two}.  
It is clear that, as $L$ becomes large, 
the number $N$ of terms in Eqn. (\ref{eqn:Ip}) 
must be increased to maintain the same accuracy.  

Table \ref{table:two} shows the difference 
between values from applying Thm. \ref{thm:two} 
to the finite-depth dispersion relation in Eqn. (\ref{eqn:finite}) 
for a given depth $L$ relative to values 
from applying Thm. \ref{thm:two} 
to the infinite-depth dispersion relation in Eqn. (\ref{eqn:fulldisp}).  
The finite-depth values clearly converge to the infinite-depth values 
as $L$ approaches $\infty$.

\begin{table}[htb]
\begin{center}
\begin{small}
\begin{tabular}{|c|l|l|l|l|} \hline
 & error & error & error & error \\ 
$N$ & $L=1\;{\rm cm}$ & $L=2\;{\rm cm}$ & $L=5\;{\rm cm}$ 
& $L=10\;{\rm cm}$\\ \hline
 5     &  6.6522e-03 & 7.9024e-03 & 8.1624e-03 &  8.1793e-03 \\
 10    &  2.1339e-03 & 6.1278e-03 & 7.9499e-03 &  8.1492e-03 \\
 20    &  6.2132e-05 & 1.8267e-03 & 6.7127e-03 &  7.9037e-03 \\
 40    &  4.9696e-07 & 5.3631e-05 & 3.0219e-03 &  6.6182e-03 \\
 80    &  3.9957e-09 & 4.5087e-07 & 1.9059e-04 &  2.9354e-03 \\
 160   &  4.2508e-11 & 3.3930e-09 & 2.0284e-06 &  1.8370e-04 \\
 320   &  7.1332e-13 & 2.5917e-11 & 1.5824e-08 &  1.9844e-06 \\ 
 640   &  1.9300e-14 & 1.9743e-13 & 1.2227e-10 &  1.5650e-08 \\
 1280  &  1.0307e-14 & 6.6570e-15 & 9.5273e-13 &  1.2160e-10 \\
\hline
\end{tabular}
\end{small}
\end{center}
\caption{
This table gives the relative error in the quantity $I_p$ 
calculated using Eqn. (\ref{eqn:Ip}) with $N$ eigenvalues 
compared to the exact expression calculated using Eqn. (\ref{eqn:Ipfp}), 
where the sum is evaluated by taking the residues at the poles of $G$.   
The errors are given for different values of the depth $L$, 
where the physical parameters are 
$\rho = 0.95\; {\rm g/cm^3}$, $\nu = 0.1\; {\rm cm^2/s}$, 
$T = 70\; {\rm g/s^2}$, $g_0 = 1000\; {\rm cm/s^2}$, 
and $\alpha = 5\; {\rm cm}^{-1}$ 
and the power spectral density described in Appendix B is used. 
}
\mylabel{table:one}
\end{table}

\begin{table}[htb]
\begin{center}
\begin{small}
\begin{tabular}{|c|l|l|l|l|l|} \hline
 & $L=0.25\;{\rm cm}$ & $L=0.5\;{\rm cm}$ & $L=1\;{\rm cm}$ 
& $L=2\:{\rm cm}$ & $L=4\;{\rm cm}$ \\ \hline
error & 1.3459e-01 & 1.4134e-02 & 9.7424e-05 & 4.4170e-09 
& 2.2693e-16 \\ \hline
\end{tabular}
\end{small}
\end{center}
\caption{
This table gives the relative error in the quantity $I_p$ 
calculated using the finite-depth dispersion relation 
in Eqn. (\ref{eqn:finite}) from Appendix D 
compared to the value 
calculated using the infinite-depth relation 
in Eqn. (\ref{eqn:fulldisp}) for various depths $L$.  
}
\mylabel{table:two}
\end{table}

\begin{rem}
\mylabel{rem:use}
It is good to compare the cost of 
performing a brute-force calculation using Eqn. (\ref{eqn:Ip}) 
versus applying Thm. \ref{thm:two}. 
When making this comparison,
one should take into account the effort of the analyst, 
not just the computational cost.  
Note that, for our system to be stable, 
it must be stable for all values of $\alpha$. 
Hence, to do a stability analysis, 
it is necessary to sweep through all values of $\alpha$ 
to find the smallest value of $\epsilon$ needed to
make some wavenumber unstable. 
For this reason, it is highly desirable to have a method
that requires little intervention by the analyst. 
When computing the sum for $I_p$ by brute force,
we need to know all of the eigenvalues, eigenfunctions, 
and adjoint eigenfunctions,
as well as how to compute the inner product to determine $\chi_{ij}$.   
Although the eigenvalues can in principle be determined 
using the dispersion relation, 
one needs to have a good initial guess for the eigenvalues 
in order to apply Newton's method.  
Though this is not, in principle, a difficult thing to do, 
it greatly complicates the sweep through the wavenumber $\alpha$.  
A considerable amount of theoretical effort is needed to compute the
quantities $\chi_{ij}$ based on Eqn. (\ref{eqn:defchi}). 

On the other hand, in order to use Thm. \ref{thm:two}, 
all that one needs is the dispersion relation for $\sigma$ 
and the poles and residues of $G$. 
In particular, there is no need to calculate 
the eigenvalues, the eigenfunctions, or the inner products $\chi_{ij}$.  
It should be noted that the ease of this method 
depends on the particular application. 
If one is interested in doing laboratory experiments 
on stochastic Faraday waves, 
then one has control over the power spectral density that is used.
In this case, one could use the power spectral density 
described in Appendix D. 
However, if one is interested in applying the results to 
a power spectral density that has been measured experimentally, 
then it is necessary to do some sort of 
rational function approximation to $G(z)$ first 
and to compute its poles subsequently. 
In this case, applying Thm. \ref{thm:two} 
would require some preliminary steps.  
\end{rem}

\section{Conclusions}

We have derived formulas that allow us to evaluate the stability of
parametrically  forced stochastic equations.  
When the stochastic forcing is multiplying a rank one matrix, 
this leads to an enormous simplification 
over a brute-force implementation of the expressions in \cite{blass_rom}, 
especially when the dimension of the system is large.  
We have shown how our results apply to a simple mechanical example, 
as well as to stochastically forced Faraday waves.

\section{Appendix A }

The  purpose of this appendix is to show how the
results in \cite{blass_rom} can be extended to apply
to systems with more general mass matrices.  That is,
we relax the assumption that the matrix ${\bf B}_0$ is Eqn. (\ref{eq:x})
is given by  ${\bf B}_0 = {\bf I}$  and
replace it with the assumption that ${\bf B}_0$ is symmetric positive
definite.   
More precisely, in this appendix, we show that if the results summarized in
\S \ref{sec:rev} hold for systems where ${\bf B}_0$ is the identity, matrix, then they 
 hold for systems where ${\bf B}_0$ is a symmetric positive definite matrix.

In order to apply the results with ${\bf B}_0 = {\bf I}$ to the
case of more general ${\bf B}_0$,  we can introduce the vector
\begin{equation}
{\bf q} = {\bf B}_0 ^{1/2} {\bf x}. 
\end{equation}
The equation for ${\bf q}$ can be written as
\begin{equation}
\odone{{\bf q}}{t} = {\bf B}_0 ^{-1/2} 
\left( {\bf A} _0 + \epsilon s(t) {\bf A} _1 \right) {\bf B}_0^{-1/2} {\bf q}.
\end{equation}
This is in the form  of Eqn. (\ref{eq:x}).
We need to know the eigenvalues of
\begin{equation}
{\bf C}  = {\bf B}_0^{-1/2} {\bf A} _0 {\bf B}_0 ^{-1/2}.
\end{equation}
It is a straightforward exercise to show that the  eigenvectors 
of ${\bf C}$ 
satisfy
\begin{equation}
{\bf q} = {\bf B}_0 ^{1/2} {\bf u}\qquad
\mbox{provided}\quad {\bf A} _0 {\bf u} = \lambda {\bf B}_0 {\bf u}.
\end{equation}
Similarly, the adjoint eigenvectors of ${\bf C} $ satisfy
\begin{equation}
{\bf p} = {\bf B}_0 ^{1/2} {\bf v}\qquad
\mbox{provided}\quad {\bf A} _0^T {\bf v} = \lambda {\bf B}_0 {\bf v}.
\end{equation}
The eigenvectors ${\bf p}$ and ${\bf q} $ should be normalized so that
\begin{equation}
\overline{{\bf p}_i}^T {\bf q} _j = \delta_{ij}.
\end{equation}
This is  so provided 
\begin{equation}
\overline{{\bf v}_i}^T {\bf B}_0 {\bf u} _j = \delta_{ij}.
\end{equation}
That is, the eigenvectors ${\bf u} _i$ and ${\bf v} _j$
must be normalized using the inner product associated with ${\bf B}_0$.  
In order to compute $\chi_{ij}$, we need to compute
\begin{equation}
\chi_{ij} =  \overline{{\bf v} _i}^T {\bf A} _1 {\bf u} _j .
\end{equation}

This completes the proof that if the results in \S \ref{sec:rev} hold
for ${\bf B}_0 ={\bf I}$, then they hold for arbitrary symmetric
positive definite ${\bf B}_0$.

\section{Appendix B }
\mylabel{sec:G}

In this appendix,  we discuss a particular choice of the function
$G(z)$ that  allows us to exercise the formulas in this paper.
In the applications we are concerned with, the function $f(t)$ whose
power spectral density we are evaluating is an acceleration.  The power spectral density of the acceleration is $\omega^4$ times the power spectral density of the displacement.  For this reason, we  require that our power spectral density 
is proportional to $\omega^4$ for small values of $\omega$.  

Using the definition of $G(z)$ in Eqn. (\ref{eq:Gdef})  and the
fact that $R(\tau)$ is the inverse Fourier transform of the
power spectral density $S(\omega)$, we see that
\begin{equation}
G(z) =  \frac{1}{2\pi} \int_0^{\infty} \infint S(\omega)  e^{i \omega \tau}
e^{ - \tau z} d \omega d \tau .
\end{equation}
Reversing the order of integration and integrating with respect to
$\tau$ gives
\begin{equation}
G(z) = \frac{1}{2 \pi  } \infint \frac{S(\omega) }{z - i \omega}
d \omega .
\mylabel{eqn:GS}
\end{equation}

We use the power spectral density
\begin{equation}
S(\omega) = \frac{ \omega^4 a^3}{\pi A_{nor} }
\left( \frac{1}{ a^8 + (\omega- \omega_0)^8 }
+ \frac{1}{ a^8 + (\omega+ \omega_0)^8 }
\right),
\mylabel{eqn:S1}
\end{equation}
\begin{equation}
A_{nor} = 
\left((1+\sqrt{2}) \omega_0^4/a^4 + 6 \omega_0^2/a^2 + 1 \right) 
\left( \sqrt{ 1 - 1/\sqrt{2} } \right) .
\end{equation}
This power spectral density has the property 
that $S(\omega) = S( - \omega)$, 
which corresponds to a real-valued autocorrelation function $R(\tau)$.  
It is normalized so that
\begin{equation}
\infint S(\omega) d \omega =1.
\end{equation}
For small values of $a$, 
this power spectral density produces narrow-band noise, 
with the energy concentrated around $\omega = \pm \omega_0$. 
For larger values of $a$, 
this power spectral density produces wide-band noise.  
With $S(\omega)$ chosen as in Eqn. (\ref{eqn:S1}), 
the function
$G(z)$ can be evaluated using contour integration.  
We get
\begin{equation}
G(z) =  \frac{1}{ 8 i \pi  a^4 A_{nor} }
\left( G_0(z,\omega_0) + G_0(z, - \omega _0) \right),
\end{equation}
where
\begin{equation}
G_0(z,\omega_0) =  \sum_{k=0}^3 
\frac{ \gamma_k \left( a \gamma_k + \omega_0 \right)^4
}{ z - i \left( a \gamma _k + \omega_0 \right) }
\;\;\;\mbox{ $\gamma_k = i e^{-3 i\pi/8} e^{ik \pi/4} $}.
\end{equation}

\section{Appendix C}

In this appendix, we show that the results in \cite{blass_rom} apply to 
systems that have a KKT-like structure. That is, to systems of the form: 
\begin{equation}
{\bf M} _0  \odone{{\bf u}}{t} = \left( {\bf K} _0 + \epsilon f(t) {\bf K} _1 \right)
{\bf u} + {\bf C} ^T {\bf p} 
\mylabel{eqn:dae1},
\end{equation}
\begin{equation}
{\bf C} {\bf u} =0.
\mylabel{eqn:dae2}
\end{equation}
Here, we  assume that ${\bf M} _0$ is symmetric positive definite.  
Among other places, systems of this type appear   when discretizing 
the equations of fluid dynamics.  Using the matrices
\begin{equation} 
{\bf B}_0 = \twomat{ {\bf M} _0}{ {\bf 0} }{ {\bf 0}^T }{0},\;\;
{\bf A} _0 = \twomat{ {\bf K} _0 }{ {\bf C} ^T}{ {\bf C} }{0},\;\;
{\bf A} _1 = \twomat{ {\bf K} _1 }{ {\bf 0} ^T}{ {\bf 0} }{0},
\mylabel{eqn:KKTB}
\end{equation}
  it is possible to write this
system as in Eqn. (\ref{eq:x}), where the mass matrix ${\bf B}_0$ is
singular.  Since the resulting mass matrix is singular, it is not
possible to directly apply the results in \S \ref{sec:rev} and  Appendix A. However, due to the special structure of this singular mass
matrix, it is possible to extend the results of  \S \ref{sec:rev} and Appendix A so 
that they apply to 
the system of equations as in Eqns. (\ref{eqn:dae1}) and (\ref{eqn:dae2}).
Thus, the purpose of this appendix
is to show that the results in \S \ref{sec:rev}  
apply to systems that have  KKT-like  structure as in  Eqns. (\ref{eqn:KKTB}).

Assuming the  matrix ${\bf C}$ has fewer rows than columns,
we can write 
\begin{equation}
{\bf C} ^T = {\bf Q} {\bf R} ,
\end{equation}
where ${\bf R}$ is a non-singular matrix and ${\bf Q}$ is a matrix whose
columns are orthonormal.
The projection matrix ${\bf P}$ can be chosen to be
a matrix whose columns are orthonormal and  orthogonal  to 
the columns of  ${\bf Q}$.  This gives us
\begin{equation}
{\bf Q}^T {\bf P} =0,
\end{equation}
\begin{equation}
{\bf P} ^T {\bf P} = {\bf I}
\end{equation}
It is clear that 
\begin{equation}
{\bf P} ^T {\bf q} =0 \implies  \mbox{ $\exists {\bf p} $  such that }
{\bf q} = {\bf C} ^T {\bf p},
\mylabel{eqn:implies}
\end{equation}
\begin{equation}
{\bf C} {\bf u} =0  \implies  \mbox{ $\exists {\bf z} $  such that }
{\bf u} = {\bf P}  {\bf z}.
\mylabel{eqn:implies2}
\end{equation}

Eqns. (\ref{eqn:implies}) and (\ref{eqn:implies2}) show that 
we can write 
Eqns. (\ref{eqn:dae1}) and (\ref{eqn:dae2}) as 
\begin{equation}
{\bf P}^T \left( {\bf M} _0 {\bf P} 
\odone{{\bf z}}{t} - \left( {\bf K} _0 + \epsilon f(t) {\bf K} _1 \right)
 {\bf P} {\bf z} \right) =0
\mylabel{eqn:reduce}
\end{equation}
The matrix ${\bf P} ^T {\bf M} _0 {\bf P} $ is non-singular; hence,
we can apply the results  for equations where ${\bf M}$ is non-singular.
We can now apply the results in \S \ref{sec:rev} to Eqn. (\ref{eqn:reduce}).  
If we do this, we get the 
 eigenvalue problem
\begin{equation}
\sigma_k {\bf P} ^T  {\bf M} _0 {\bf P} =  {\bf P} ^T{\bf A} _0 {\bf P} {\bf z}_k .
\end{equation}
If ${\bf z}_k$ is an eigenvector of this eigenvalue problem, then we can
define ${\bf u} _k = {\bf P} {\bf z} _k$, and ${\bf p} _k$ such that
\begin{equation}
\sigma_k {\bf M} _0 {\bf u}_k = {\bf K} _0 {\bf u}_k + {\bf C} ^T {\bf p}i_k,
\end{equation}
\begin{equation}
{\bf C} {\bf u}_k =0,
\end{equation}
\begin{equation}
\leig _k = \twovec{ {\bf P} {\bf z} _k }{ {\bf p} _k}= \twovec{ {\bf u} _k}{{\bf p} _k}.
\end{equation}
Thus, the vector $\leig_k$
is an eigenvector of
the eigenvalue problem in Eqn. (\ref{eqn:geneig}) with the
 matrices defined as in Eqn. (\ref{eqn:KKTB}).   
A similar result holds for the adjoint eigenvector $\reig_k$.  
It is straightforward to show that the eigenvectors and adjoint eigenvectors
based on Eqn. (\ref{eqn:reduce})  satisfy the same normalization conditions as in
\S \ref{sec:rev} and  yield the same constant $\chi_{ij}$
provided we use the matrices defined as in Eqn. (\ref{eqn:KKTB}).

\section{Appendix D}

In this appendix, we  give the dispersion  relation for
the eigenvalues associated with capillary gravity waves 
as discussed in \S \ref{sec:far}.  
In particular, we consider Eqns. (\ref{eqn:momx0}), (\ref{eqn:momz0}),
and (\ref{eqn:cont}), along with the boundary conditions
in Eqns. (\ref{eqn:bcbot}), (\ref{eqn:kinem}), (\ref{eqn:bc1}), and  (\ref{eqn:bc2}).
If we  assume the  temporal  dependence 
of $(u,w,p,h)$ is of the form 
$e^{\sigma t}$, we get an eigenvalue problem for $\sigma$.  
In this appendix, we
give the   one-dimensional transcendental equation for the
eigenvalues $\sigma$  of this system of equations.

In writing down the dispersion relation, we use the parameter
\begin{equation}
\tau = T \alpha^2 + \rho g_0 .
\mylabel{eqn:tau}
\end{equation}
We can write the dispersion relation for $\sigma$ as 
$H(\sigma)=0$, where 
\begin{align}
  H(\sigma) = d_0 + \alpha d_s \sinh (\alpha L)\sin(\beta L)
+\beta d_c \cosh(\alpha L)\cos(\beta L) + 
\notag\\  \tau d_\tau 
(\alpha \cosh(\alpha L)\sin(\beta L)-\beta \sinh(\alpha L)\cos(\beta L))
  \mylabel{disper}
\end{align}
and 
\begin{subequations}\mylabel{constants}
\begin{align}
&  d_0 = -4\alpha^2\beta\nu(2\alpha^2 \nu+\sigma)\\
& d_c = \sigma^2+4\alpha^2\nu \sigma + 8 \alpha^4\nu^2\\
& d_s = -\sigma^2-8\alpha^2\nu \sigma - 8 \alpha^4\nu^2
=-d_c-4\alpha^2\nu \sigma\\
& d_\tau = -\frac{\alpha}{\rho}
\end{align}
\end{subequations}
and 
\begin{equation} 
\beta = i\sqrt{\frac{\sigma}{\nu}+\alpha^2} .
\end{equation} 

In order to use the results of Thm. \ref{thm:two}, we write this as
\begin{equation}
H(\sigma) = H_0(\sigma) +  \rho g_0 H_1(\sigma) ,
\end{equation}
where 
\begin{align}
  H_0(\sigma) = d_0 + \alpha d_s \sinh (\alpha L)\sin(\beta L)
+\beta d_c \cosh(\alpha L)\cos(\beta L) +
\notag\\ T \alpha^2  d_\tau 
(\alpha \cosh(\alpha L)\sin(\beta L)-\beta \sinh(\alpha L)\cos(\beta L)),
  \mylabel{disper2}
\end{align}
\begin{equation}
  H_1(\sigma) = 
 d_\tau 
(\alpha \cosh(\alpha L)\sin(\beta L)-\beta \sinh(\alpha L)\cos(\beta L)).
\end{equation}
Dividing by $ \rho H_1(\sigma)$, we get the dispersion relation
\begin{equation}
f(\sigma) = f_A(\sigma) + g_0,
\end{equation}
where
\begin{equation}
f_A(\sigma) = \frac{ H_0(\sigma) }{\rho H_1(\sigma ) }.
\mylabel{eqn:finite}
\end{equation}

\bibliography{stoch_03}
\end{document}

%% file: ss.tex
\newcommand{\pdone}[2]{\frac{\partial #1}{\partial #2}}

\newcommand{\odone}[2]{\frac{d #1}{d #2}}
\newcommand{\odtwo}[2]{\frac{d ^2#1}{d #2^2}}

\newcommand{\twovec}[2]{ \left( \begin{array}{c} #1 \\ #2 \end{array} \right) }

\newcommand{\twomat}[4]{ \left( \begin{array}{cc}  #1 & #2 
\\#3 & #4 \end{array} \right)}

\newcommand{\infint}{\int_{- \infty}^{\infty}}

\newcommand{\hatphi}{\hat{\phi}}

\newcommand{\hatsig}{\hat{\sigma}}

\newcommand{\dotthe}{\dot{\theta}}

\newcommand{\dotxi}{\dot{\xi}}